\documentclass[a4paper,11pt]{amsart}
\usepackage{amsmath,amsthm,amssymb}
\usepackage{color}
\newtheorem{Theorem}{Theorem}[section]

\newtheorem{Corollary}[Theorem]{Corollary}
\newtheorem{Lemma}[Theorem]{Lemma}
\newtheorem{Proposition}[Theorem]{Proposition}

\begin{document}

%{{{ Title etc..
%{{{ Title
\title
[On gwp for nlsr in scaling subcritical and critical cases]
{On global well-posedness for nonlinear semirelativistic equations
in some scaling subcritical and critical cases}
%}}}

 %{{{ Author
 %{{{ K. Fujiwara
 \author[K. Fujiwara]{Kazumasa Fujiwara}

\address{
Department of Pure and Applied Physics \\ Waseda University \\
3-4-1, Okubo, Shinjuku-ku, Tokyo 169-8555 \\ Japan}

\email{k-fujiwara@asagi.waseda.jp}

\thanks{
The first author was supported in part by the Japan Society for the Promotion of Science,
Grant-in-Aid for JSPS Fellows no 16J30008
and Top Global University Project, Waseda University.
}
%}}}

%{{{ V. Georgiev----------Author 1
\author[V. Georgiev]{Vladimir Georgiev}

\address{
Department of Mathematics \\
University of Pisa \\
 Largo Bruno Pontecorvo 5
 I - 56127 Pisa
 \\ Italy \\
 and \\
 Faculty of Science and Engineering \\ Waseda University \\
 3-4-1, Okubo, Shinjuku-ku, Tokyo 169-8555 \\ 
Japan}

\email{georgiev@dm.unipi.it}

\thanks{ The second author was supported in part by Contract FIRB
" Dinamiche Dispersive: Analisi di Fourier e Metodi Variazionali",
2012, by INDAM, GNAMPA - Gruppo Nazionale per l'Analisi Matematica, la Probabilita e le loro Applicazion and by Institute of Mathematics and Informatics, Bulgarian Academy of Sciences and Top Global University Project, Waseda University.}
%}}}

%{{{ T. Ozawa
 \author[T. Ozawa]{Tohru Ozawa}

\address{
Department of Applied Physics \\ Waseda University \\
3-4-1, Okubo, Shinjuku-ku, Tokyo 169-8555 \\ Japan}

\email{txozawa@waseda.jp}
\thanks{The third author was supported in part by
Grant-in-Aid for Scientific Research (A) Number 26247014.}
%}}}
%}}}

%}}}



\maketitle

%{{{ abstract
\begin{abstract}
In this paper, the global well-posedness of semirelativistic
equations with a power type nonlinearity on Euclidean spaces is studied.
In two dimensional $H^s$ scaling subcritical case with $1 \leq s \leq 2$,
the local well-posedness follows from a Strichartz estimate.
In higher dimensional $H^1$ scaling subcritical case,
the local well-posedness for radial solutions follows
from a weighted Strichartz estimate.
Moreover, in three dimensional $H^1$ scaling critical case,
the local well-posedness for radial solutions
follows from a uniform bound of solutions which may be derived by
the corresponding one dimensional problem.
Local solutions may be extended by a priori estimates.
\end{abstract}
%}}}

%{{{ section: Intro
\section{Introduction}
%{{{ Intro
In this paper,
we consider the well-posedness of the following Cauchy problem:
	\begin{align}
	\begin{cases}
	i \partial_t u - (- \Delta)^{1/2} u = - i |u|^{p-1} u,
	&\quad t \in \mathbb R,\quad x \in \mathbb R^n,\\
	u(0) = u_0, &\quad x \in \mathbb R^n,
	\end{cases}
	\label{eq:1}
	\end{align}
where $n \geq 1$, $p > 1$, $\Delta$ is the Laplacian,
and $(- \Delta)^{1/2} = \mathfrak F^{-1} |\xi| \mathfrak F$
with the Fourier transform $\mathfrak F$.
%}}}

%{{{ Similar models can be connected with the simulations of neuroscience processes.
Similar models can be connected with the simulations of neuroscience processes.
Typical one is  the cyclical alternation of REM (rapid eye movement)
and NREM (non-rapid eye movement) sleep.
See \cite{ACGM}.
The model of  the alternation of REM and NREM sleep is starting
from the classical Kuramoto model \cite{Kur},
having its origin in special type of  Landau - Ginzburg model
	\begin{equation}
	i \partial_t u - \mathcal{H}  u = - i Q(u), 
	\label{eq:2}
	\end{equation}
where $\mathcal{H}$ is appropriate Hamiltonian operator
and $Q(u)$ is appropriate cubic type nonlinearity.
In this work we substitute the specific cubic nonlinearity $Q(u)$
by a self - interacting nonlinear term
$|u|^{p-1} u$ and our goal is to implement the recent development
of fractional quantum mechanical approach
(see \cite{La00} ) based on the choice of
$\mathcal{H} = D = (- \Delta)^{1/2}$ as a Hamiltonian of the process.
%}}}

%{{{ We shall observe some new interesting phenomena.
We shall observe some new interesting phenomena.
On one hand, the contraction of some Sobolev norms of the solutions to \eqref{eq:1}
is manifested only for positive time $t>0$,
and therefore we have a similarity to a diffusion type process.
%}}}

%{{{ The Cauchy problem for \eqref{eq:1} has different conserved (or bounded)
The Cauchy problem for \eqref{eq:1} has different conserved (or bounded)
quantities that can be compared with the classical NLS with self interaction term
	\begin{align}
	\begin{cases}
	i \partial_t u - (- \Delta)^{1/2} u = -  |u|^{p-1} u,
	&\quad t \in \mathbb R,\quad x \in \mathbb R^n,\\
	u(0) = u_0, &\quad x \in \mathbb R^n,
	\end{cases}
	\label{eq:3}
	\end{align}
Indeed, natural Sobolev norm that can be controlled for \eqref{eq:3}
is $H^{1/2}(\mathbb{R}^n)$,
while \eqref{eq:1} enables one to control $H^{1}(\mathbb{R}^n)$ norm
but only in the future time instants $t>0$.
%}}}

%{{{ To state our main result, we turn to the introduction of the notations used below.
To state our main result, we turn to the introduction of the notations used below.
For a Banach space $X$ and $1 \leq p \leq \infty$
let $L^p(\mathbb R^n;X)$ be a $X$-valued Lebesgue space of $p$-th power.
We abbreviate $L^p(\mathbb R^n;\mathbb C)$ as $L^p(\mathbb R^n)$.
For $f,g \in L^2(\mathbb R^n)$, we define a inner product as
	\[
	\langle f , g \rangle_{L^2(\mathbb R^n)}
	= \int_{\mathbb R^n} f(x) \overline g(x) dx.
	\]
For $s \in \mathbb R$,
let $H^s(\mathbb R^n)$ be the usual inhomogeneous Sobolev space
defined as $H^s(\mathbb R^n) = (1-\Delta)^{-s/2} L^2(\mathbb R^n)$.
Let $\dot H^s(\mathbb R^n)$ be the usual homogeneous Sobolev space
defined as $\dot H^s(\mathbb R^n) = (-\Delta)^{-s/2} L^2(\mathbb R^n)$.
For $f,g: A \to \lbrack 0,\infty)$ with a set $A$,
$f \lesssim g$ means
there exists $C > 0$ for any $a \in A$ such that $f(a) \leq C g(a)$.
For Banach spaces $X,Y$,
$Y \hookrightarrow X$ means
$Y \subset X$ with continuous embedding.
Moreover, we say a Cauchy problem is locally well-posed in $X$,
if for any $X$-valued initial data,
there exists $T > 0$ and a Banach space $Y \hookrightarrow C([0,T];X)$
such that there is a unique solution to the Cauchy problem in $Y$
and $\| u_n - u \|_Y \to 0$ as $\|u_{0,n} - u_0\|_X \to 0$,
where $u_n$ and $u$ are solutions for the Cauchy problem
for initial data $u_0$ and $u_{0,n}$, respectively.
We also say a Cauchy problem is globally well-posed in $X$
if the Cauchy problem is locally well-posed for any $T > 0$.
Moreover, we also say a Cauchy problem is globally well-posed in $X$
with sufficiently small data,
if we have the property above for
sufficiently small $X$-valued data.
%}}}

%{{{ \eqref{eq:1} with is invariant under the scale transformation
\eqref{eq:1} with is invariant under the scale transformation
	\[
	u_\lambda(t,x) = \lambda^{1/(p-1)} u(\lambda t, \lambda x)
	\]
with $\lambda >0$.
Then
	\[
	\|u_{0,\sigma}\|_{\dot H^s (\mathbb R^n)}
	= \sigma^{1/(p-1)+s-n/2} \|u_0\|_{\dot H^s(\mathbb R^n)}
	\]
and with
	\[
	s = s_{n,p} := n/2 - 1/(p-1) < n/2,
	\]
$\dot H^s$ norm of initial data is also invariant.
$s_{n,p}$ is called scale critical exponent.
We also call $p_{n,s} = 1 + 2/(n-2s)$ the $H^s(\mathbb R^n)$ scaling critical power.
For any $s$, in the scaling subcritical case where $p < p_{n,s}$,
\eqref{eq:1} is expected to have local solution for
any $H^s(\mathbb R^n)$ initial data
on the analogy of scaling invariant Schr\"odinger equation.
For instance, we refer the reader \cite{C,CW1,CW2,II,IW}.
However, with power type nonlinearity without gauge invariance,
semirelativistic equations may not be locally well-posed
even in scaling subcritical case,
see \cite{F}.
%}}}

%{{{ In this paper,
In this paper,
we show the following global well-posedness for the Cauchy problem \eqref{eq:1}:
%}}}

%{{{ Proposition : 1D
\begin{Proposition}
\label{Proposition:1.1}
Let $n=1$.
For $p > 1$
the Cauchy problem \eqref{eq:1} is globally well-posed
in $H^1(\mathbb R^1)$.
Moreover, for $p = 3$,
the Cauchy problem \eqref{eq:1} is globally well-posed
in $H^s(\mathbb R^1)$ with $1 < s \leq 2$.
\end{Proposition}
%}}}

%{{{ Proposition : 2D
\begin{Proposition}
\label{Proposition:1.2}
Let $n=2$.
For $p > 1$ and $3/4 < s < p$,
the Cauchy problem \eqref{eq:1} is locally well-posed
in $H^s(\mathbb R^2)$.
Moreover, for $p > 1$,
the Cauchy problem \eqref{eq:1} is globally well-posed
in $H^1(\mathbb R^2)$.
For $p = 3$,
the Cauchy problem \eqref{eq:1} is globally well-posed
in $H^s(\mathbb R^2)$ with $1 < s \leq 2$.
\end{Proposition}
%}}}

%{{{ Proposition : MD
\begin{Proposition}
\label{Proposition:1.3}
Let $n \geq 3$ and $u_0$ be radial.
For $1 < p < p_{n,1} = 1 + \frac{2}{n-2}$,
the Cauchy problem \eqref{eq:1} is globally well-posed
in $H_{\mathrm{rad}}^1(\mathbb R^3)$.
\end{Proposition}
%}}}

%{{{ Proposition : 3D
\begin{Proposition}
\label{Proposition:1.4}
Let $n=3$ and $u_0$ be radial.
For $p = p_{3,1} = 3$,
the Cauchy problem of \eqref{eq:1} is globally well-posed in
$H_{\mathrm{rad}}^1(\mathbb R^3)$
with sufficiently small $H_{\mathrm{rad}}^1(\mathbb R^3)$ data.
\end{Proposition}
%}}}

%{{{ For three dimensional case $p=3$ is a critical value
For three dimensional case $p=3$ is a critical value in view of the result in \cite{I}.
However, the result in \cite{I} treats nongauge invariant nonlinearities
having constant sign,
for which the test function method works.
The question of the existence of local and global solutions
for $n \geq 3$ and $p \geq 1+ 2/(n-2)$ seems still open.
%}}}

%{{{ This article is composed as follows:
This paper is organized as follows:
In section 2, we collect a priori estimates for \eqref{eq:1}.
In section 3, we prove Propositions
\ref{Proposition:1.1},
\ref{Proposition:1.2},
\ref{Proposition:1.3},
and \ref{Proposition:1.4}.
In one dimensional case,
local well-posedness follows from
a standard contraction argument.
In the case where $n=2$,
local well-posedness follows from
the Strichartz estimate derived by Nakamura and one of the authors in \cite{NO}.
However, with this Strichartz estimate,
we may control solutions uniformly only
in the $H^s(\mathbb R^n)$ setting with $s > (n+1)/4$.
We remark that it seems difficult
to obtain the local well-posedness if $s \leq (n+1)/4$
by a simple application of an improved Strichartz estimate for radial solutions
derived by Guo and Wang in \cite{GW}.
In the case where $n \geq 3$,
a weighted Sobolev space derived by Bellazzini, Visciglia, and one of the author
in \cite{BGV},
and uniform controls derived by Sickel and Skrzypczak in \cite{SS}
(see also \cite{CO})
play an critical role to prove local well-posedness.
Moreover, in the $3$ dimensional scaling critical case where $p=3$,
we obtain a uniform control of solutions by transforming \eqref{eq:1}
into the corresponding $1$ dimensional problem.
%}}}
%}}}

%{{{ section: A priori estimates
\section{A priori estimates}
\label{section:2}
%{{{ Preliminary Introduction
Here we collect some a priori estimates.
The Cauchy problem \eqref{eq:1}
is rewritten as the following integral equation:
	\begin{align}
	u(t) = U (t) u_0 - \int_0^t U (t-t') |u(t')|^{p-1} u(t') dt',
	\label{eq:4}
	\end{align}
where $U(t) = e^{-itD}$ and $D = (-\Delta)^{1/2}$.
%}}}

%{{{ Proposition: L^2 estimate
\begin{Proposition}
\label{Proposition:2.1}
Let $n \in \mathbb N$ and $p > 1$.
Let $u_0 \in L^{2}(\mathbb R^n)$ and $T>0$.
Let $u \in L^\infty(0,T;L^2(\mathbb R^n)) \cap L^p(0,T; L^{2p}(\mathbb R^n))$
be a solution to the integral equation \eqref{eq:4}
for the initial data $u_0$.
Then, for any $t_1, t_2$ with $0 < t_1 < t_2 < T$,
	\[
	\|u(t_2)\|_{L^2(\mathbb R^n)}^2
	+ 2 \|u\|_{L^{p+1}(t_1,t_2;L^{p+1}(\mathbb R^n))}^{p+1}
	= \|u(t_1)\|_{L^2(\mathbb R^n)}^2.
	\]
\end{Proposition}

\begin{proof}
A formal computation yields immediately the proposition.
However, actual proof requires some regularization procedure
to justify the formal calculation.
Here we give a direct proof based on the integral equation
on the basis of the method in \cite{O}.
\begin{align*}
	&\langle u(t_2),u(t_2)\rangle_{L^2(\mathbb R^n)}\\
	&=
	\bigg\langle U(t_2-t_1) u(t_1)
	- \int_{t_1}^{t_2} \hspace*{-10pt} U(t_2-t) |u(t)|^{p-1} u(t) dt, u(t_2)
	\bigg\rangle_{L^2(\mathbb R^n)}\\
	&= \|u(t_1)\|_{L^2(\mathbb R^n)}^2
	- 2 \mathrm{Re} \bigg\langle U(t_2-t_1) u(t_1),
	\int_{t_1}^{t_2} \hspace*{-10pt} U(t_2-t) |u(t)|^{p-1} u(t) dt
	\bigg\rangle_{L^2(\mathbb R^n)}\\
	& + \bigg\langle \int_{t_1}^{t_2} U(t_2-t) |u(t)|^{p-1} u(t) dt
	, \int_{t_1}^{t_2} U(t_2-t') |u(t')|^{p-1} u(t') dt' \bigg\rangle_{L^2(\mathbb R^n)}\\
	&= \|u(t_1)\|_{L^2(\mathbb R^n)}^2
	- 2 \mathrm{Re} \bigg\langle U(t_2-t_1) u(t_1),
	\int_{t_1}^{t_2} \hspace*{-10pt} U(t_2-t) |u(t)|^{p-1} u(t) dt
	\bigg\rangle_{L^2(\mathbb R^n)}\\
	& + 2 \mathrm{Re} \int_{t_1}^{t_2} \bigg\langle |u(t)|^{p-1} u(t),
	\int_{t_1}^{t} U(t-t') |u(t')|^{p-1} u(t') dt' \bigg\rangle_{L^2(\mathbb R^n)} dt\\
	&= \|u(t_1)\|_{L^2(\mathbb R^n)}^2
	- 2 \mathrm{Re} \bigg\langle U(t_2-t_1) u(t_1),
	\int_{t_1}^{t_2} \hspace*{-10pt} U(t_2-t) |u(t)|^{p-1} u(t) dt
	\bigg\rangle_{L^2(\mathbb R^n)}\\
	& + 2 \mathrm{Re} \int_{t_1}^{t_2} \big\langle |u(t)|^{p-1} u(t),
	U(t-t_1) u(t_1) - u(t) \big\rangle_{L^2(\mathbb R^n)} dt\\
	&= \|u(t_1)\|_{L^2(\mathbb R^n)}^2
	- 2 \|u\|_{L^{p+1}(t_1,t_2;L^{p+1}(\mathbb R^n))}^{p+1}.
	\end{align*}
\end{proof}
%}}}

%{{{ Proposition: Energy Estimate
\begin{Proposition}
\label{Proposition:2.2}
Let $n \in \mathbb N$ and $p>1$.
Let $u_0 \in H^{1}(\mathbb R^n)$ and $T>0$.
Let $u \in L^\infty(0,T; H^{1}(\mathbb R^n)) \cap L^{p-1}(0,T; L^\infty(\mathbb R^n))$
be a solution to the integral equation \eqref{eq:4} for the initial data $u_0$.
Then, for any $t_1, t_2$ with $0 \leq t_1 < t_2 \leq T $,
	\begin{align}
	&\|\nabla u(t_2)\|_{L^2(\mathbb R^n)}^2
	+ 2 \| |u|^{\frac{p-1}{2}} \nabla u\|_{L^2(t_1,t_2;L^2(\mathbb R^n))}^2
	\nonumber \\
	&+ \frac{p-1}{2}
	\| |u|^{\frac{p-3}{2}} \nabla |u|^2 \|_{L^2(t_1,t_2;L^2(\mathbb R^n))}^2
	\nonumber \\
	&= \|\nabla u(t_1)\|_{L^2(\mathbb R^n)}^2.
	\label{eq:5}
	\end{align}
\end{Proposition}

\begin{proof}
Since $|u|^{p-1} u \in L^1(0,T;H^1(\mathbb R^n))$,
	\begin{align*}
	&\|\nabla u(t_2)\|_{L^{2}(\mathbb R^n)}^2\\
	&= \|\nabla u(t_1)\|_{L^2(\mathbb R^n)}^2
	- 2 \mathrm{Re} \int_{t_1}^{t_2} \big\langle \nabla (|u(t)|^{p-1} u(t)),
	\nabla u(t) \big\rangle_{L^2(\mathbb R^n)} dt\\
	&= \|\nabla u(t_1)\|_{L^2(\mathbb R^n)}^2
	- 2 \mathrm{Re} \int_{t_1}^{t_2} \big\langle \nabla |u(t)|^{p-1},
	\overline {u(t)} \nabla u(t) \big\rangle_{L^2(\mathbb R^n)} dt\\
	&- 2 \int_{t_1}^{t_2} \big\langle |u(t)|^{p-1} \nabla u(t),
	\nabla u(t) \big\rangle_{L^2(\mathbb R^n)} dt\\
	&= \|\nabla u(t_1)\|_{L^2(\mathbb R^n)}^2
	- \frac{p-1}{2} \int_{t_1}^{t_2} \big\langle |u(t)|^{p-3} \nabla |u(t)|^2,
	\nabla |u(t)|^{2} \big\rangle_{L^2(\mathbb R^n)} dt\\
	&- 2 \int_{t_1}^{t_2} \big\langle |u(t)|^{p-1} \nabla u(t),
	\nabla u(t) \big\rangle_{L^2(\mathbb R^n)} dt.
	\end{align*}
\end{proof}
%}}}

%{{{ Proposition : Energy Estimate 3
\begin{Proposition}
\label{Proposition:2.3}
Let $n=1,2$, $p>1$, $n/2 < s < \min(2,p)$, and $T>0$.
Let $u_0 \in H^s(\mathbb R^n)$
and $u \in L^\infty(0,T;H^s(\mathbb R^n)) \cap L^2(0,T;L^\infty(\mathbb R^n))$
be a solution to \eqref{eq:4}
for the initial data $u_0$.
Then for any $t_1, t_2$ with $0 < t_1 < t_2 < T$,
	\[
	\|u(t_2)\|_{\dot H^s(\mathbb R^n)}^2
	\leq \|u(t_1)\|_{\dot H^s(\mathbb R^n)}^2
	+ C \int_{t_1}^{t_2} \|u(t)\|_{L^{\infty}(\mathbb R^n)}^{p-1}
	\| u(t) \|_{\dot H^s(\mathbb R^n)}^2 dt.
	\]
\end{Proposition}

\begin{proof}
	\begin{align*}
	&\|u(t_2)\|_{\dot H^s(\mathbb R^n)}^2\\
	&= \|u(t_1)\|_{\dot H^s(\mathbb R^n)}^2
	- 2 \mathrm{Re} \int_{t_1}^{t_2} \langle D^s ( |u(t)|^{p-1} u(t) ),
	D^s u(t) \rangle_{L^2(\mathbb R^n)} dt\\
	&\leq \|u(t_1)\|_{\dot H^s(\mathbb R^n)}^2
	+ 2 \int_{t_1}^{t_2} \| D^s ( |u(t)|^{p-1} u(t) ) \|_{L^2(\mathbb R^n)}
	\| u(t) \|_{\dot H^s(\mathbb R^n)} dt\\
	&\leq \|u(t_1)\|_{\dot H^s(\mathbb R^n)}^2
	+ C \int_{t_1}^{t_2} \|u(t)\|_{L^{\infty}(\mathbb R^n)}^{p-1}
	\| u(t) \|_{\dot H^s(\mathbb R^n)}^2 dt,
	\end{align*}
where we use the nonlinear estimate
	\[
	\| |f|^{p-1} f \|_{\dot H^s(\mathbb R^n)}
	\lesssim \| f\|_{L^\infty(\mathbb R^n)}^{p-1} \| f\|_{\dot H^s(\mathbb R^n)}
	\]
(see \cite[Lemma 3.4]{GOV}).
\end{proof}
%}}}

%{{{ Proposition: Energy Estimate 2
\begin{Proposition}
\label{Proposition:2.4}
Let $ 1 \leq n \leq 3$, $u_0 \in H^2(\mathbb R^n)$ and $T>0$.
Let $u \in C((0,T); H^{2}(\mathbb R^n) \cap L^\infty(\mathbb R^n))$
be a solution to the integral equation \eqref{eq:4} for the initial data $u_0$.
Then, for any $t_1, t_2$ with $0 < t_1 < t_2 < T ,$
	\begin{align}
	&\|u(t_2)\|_{\dot H^{2}(\mathbb R^n)}^2
	+ 2 \sum_{j,k=1}^n \int_{t_1}^{t_2}
	\| u(t) \partial_j \partial_k u(t) \|_{L^2(\mathbb R^n)}^2 dt
	\nonumber\\
	&\leq \|u(t_1)\|_{\dot H^2(\mathbb R^n)}^2
	+ 2 n^2 (n+1)
	\int_{t_1}^{t_2}
	\|u(t)\|_{\dot H^1(\mathbb R^n)}^{4-n} \| u(t)\|_{\dot H^2(\mathbb R^n)}^{n} dt.
	\label{eq:6}
	\end{align}
\end{Proposition}

\begin{proof}
Since $|u|^2 u \in C((0,T);H^2(\mathbb R^n))$,
the following calculation is justified by the Plancherel identity:
\begin{align*}
	&\|u(t_2)\|_{\dot H^{2}(\mathbb R^n)}^2\\
	&= \|u(t_1)\|_{\dot H^2(\mathbb R^n)}^2
	- 2 \mathrm{Re}
	\int_{t_1}^{t_2} \big\langle \Delta |u(t)|^2 u(t),
	\Delta u(t) \big\rangle_{L^2(\mathbb R^n)} dt\\
	&= \|u(t_1)\|_{\dot H^2(\mathbb R^n)}^2
	- 2 \mathrm{Re} \sum_{j,k=1}^n
	\int_{t_1}^{t_2} \big\langle |u(t)|^2 \partial_j \partial_k u(t),
	\partial_j \partial_k u(t) \big\rangle_{L^2(\mathbb R^n)} dt\\
	&- 4 \mathrm{Re} \sum_{j,k=1}^n
	\int_{t_1}^{t_2} \big\langle \partial_k u(t) \partial_j |u(t)|^2 ,
	\partial_j \partial_k u(t) \big\rangle_{L^2(\mathbb R^n)} dt\\
	&- 2 \mathrm{Re} \sum_{j,k=1}^n
	\int_{t_1}^{t_2} \big\langle \partial_j \partial_k |u(t)|^2 ,
	\overline {u(t)} \partial_j \partial_k u(t) \big\rangle_{L^2(\mathbb R^n)} dt\\
	&= \|u(t_1)\|_{\dot H^2(\mathbb R^n)}^2
	- 2 \sum_{j,k=1}^n \int_{t_1}^{t_2}
	\| u(t) \partial_j \partial_k u(t) \|_{L^2(\mathbb R^n)}^2 dt\\
	&+ 2 \sum_{j,k=1}^n \int_{t_1}^{t_2} \big\langle \partial_j^2 |u(t)|^2 ,
	|\partial_k u(t)|^2 \big\rangle_{L^2(\mathbb R^n)} dt\\
	&- \sum_{j,k=1}^n
	\int_{t_1}^{t_2} \big\langle \partial_j \partial_k |u(t)|^2 ,
	\partial_j \partial_k |u(t)|^2
	- 2 \mathrm{Re} ( \overline{\partial_j u(t)} \partial_k u(t) )
	\big\rangle_{L^2(\mathbb R^n)} dt.
	\end{align*}
By the Young inequality,
	\begin{align*}
	& \sum_{j,k=1}^n
	\big\langle \partial_j^2 |u(t)|^2, |\partial_k u(t)|^2
	\big\rangle_{L^2(\mathbb R^n)}\\
	&\leq \sum_{j,k=1}^n \| \partial_j^2 |u(t)|^2 \|_{L^2(\mathbb R^n)}
	\| \partial_k u(t) \|_{L^4(\mathbb R^n)}^2\\
	& \leq \frac 1 4 \sum_{j=1}^n \| \partial_j^2 |u(t)|^2 \|_{L^2(\mathbb R^n)}^2
	+ n^2 \sum_{k=1}^n \| \partial_k u(t) \|_{L^4(\mathbb R^n)}^4\\
	&\leq \frac 1 4 \sum_{j,k=1}^n
	\| \partial_j \partial_k |u(t)|^2 \|_{L^2(\mathbb R^n)}^2
	+ n^2 \sum_{k=1}^n \| \partial_k u(t) \|_{L^4(\mathbb R^n)}^4.
	\end{align*}
Similarly,
	\begin{align*}
	&\sum_{j,k=1}^n \big\langle \partial_j \partial_k |u(t)|^2 ,
	\mathrm{Re} ( \overline{\partial_j u(t)} \partial_k u(t) )
	\big\rangle_{L^2(\mathbb R^n)}\\
	&\leq \frac 1 4 \sum_{j,k=1}^n
	\| \partial_j \partial_k |u(t)|^2 \|_{L^2(\mathbb R^n)}^2
	+ \sum_{j,k=1}^n \| \partial_j u(t) \|_{L^4(\mathbb R^n)}^2
	\| \partial_k u(t) \|_{L^4(\mathbb R^n)}^2.
	\end{align*}
Therefore, by the Sobolev inequality,
	\begin{align*}
	&\|u(t_2)\|_{\dot H^{2}(\mathbb R^n)}^2\\
	&\leq \|u(t_1)\|_{\dot H^2(\mathbb R^n)}^2
	- 2 \sum_{j,k=1}^n \int_{t_1}^{t_2}
	\| u(t) \partial_j \partial_k u(t) \|_{L^2(\mathbb R^n)}^2 dt\\
	&+ 2 n^2 \sum_{k=1}^n \int_{t_1}^{t_2} \| \partial_k u(t) \|_{L^4(\mathbb R^n)}^4 dt
	+ 2 \sum_{j,k=1}^n \int_{t_1}^{t_2} \| \partial_j u(t) \|_{L^4(\mathbb R^n)}^2
	\| \partial_k u(t) \|_{L^4(\mathbb R^n)}^2 dt\\
	&\leq \|u(t_1)\|_{\dot H^2(\mathbb R^n)}^2
	- 2 \sum_{j,k=1}^n \int_{t_1}^{t_2}
	\| u(t) \partial_j \partial_k u(t) \|_{L^2(\mathbb R^n)}^2 dt\\
	&+ 2 n^2(n+1) \int_{t_1}^{t_2}
	\|u(t)\|_{\dot H^1(\mathbb R^n)}^{4-n} \| u(t)\|_{\dot H^2(\mathbb R^n)}^{n} dt.
	\end{align*}
\end{proof}

%}}}

%}}}

%{{{ section: Proof
\section{Proof of the Propositions}
%{{{ subsection: 1d case
\subsection{1 dimensional case}
Since $L^\infty(\mathbb R) \hookrightarrow H^1(\mathbb R) \hookrightarrow H^2(\mathbb R)$,
Proposition \ref{Proposition:1.1} is obtained by a standard contraction argument
with the Sobolev inequality and a priori estimates
Propositions \ref{Proposition:2.1}, \ref{Proposition:2.2},
\ref{Proposition:2.3}, and \ref{Proposition:2.4}.
%}}}

%{{{ subsection: 2d case
\subsection{2 dimensional case}
%{{{ In two dimensional case,
In two dimensional case,
the local well-posedness may be obtained
by the following Strichartz estimates:
%}}}

%{{{ Lemma : Strichartz
\begin{Lemma}[{\cite[Lemma 2.1]{NO}}, {\cite[Remark 3.2]{GV}}]
\label{Lemma:3.1}
Let $n = 1,2,3$.
Let
	\[
	\alpha(r) = \frac{1}{2} - \frac{1}{r},\quad
	\lambda = \frac{n+1}{2}, \quad
	\sigma = n - 1.
	\]
Let $(q_1,r_1)$ and $(q_2,r_2)$ satisfy
	\begin{align}
	\frac{1}{r_j} = \frac{1}{2} - \frac{2}{\sigma} \frac{1}{q_j}
	\label{eq:7}
	\end{align}
and $2 \leq r_j \leq \infty$ if $n = 1,2$
and $ 2 \leq r_j < \infty$ if $ n = 3 $ for $j=1,2$.
Then for $s \in \mathbb R$,
	\begin{align*}
	\|U(t) \phi\|_{L^{q_1}(0,T; B_{r_1}^{s - \lambda \alpha(r_1)}(\mathbb R^n))}
	&\lesssim \|\phi\|_{H^s(\mathbb R^n)},\\
	\bigg\| \int_0^t U(t-t') h(t') dt' \bigg\|_{L^{q_1}(0,T; B_{r_1}^{s - \lambda \alpha(r_1)}(\mathbb R^n))}
	&\lesssim \| h \|_{L^{q_2'}(0,T; B_{r_2'}^{s - \lambda \alpha(r_2')}(\mathbb R^n))},
	\end{align*}
where $B_{p,q}^s(\mathbb R^n)$ is the usual inhomogeneous Besov space.
\end{Lemma}
%}}}

%{{{ Lemma : embedding
\begin{Lemma}
\label{Lemma:3.2}
Let $r > 2$, and $T>0$.
If $q \geq 4$ and
	\[
	s > \frac{3}{4} + \frac{1}{2r},
	\]
then
$L^q(0,T;B_{r}^{s- 3/2 \cdot \alpha(r)}(\mathbb R^2))
\hookrightarrow L^4(0,T;L^\infty(\mathbb R^2))$.
\end{Lemma}

\begin{proof}
Since
	\[
	s - \frac{3}{2} \alpha(r) - \frac{2}{r} > 0
	\Longleftrightarrow
	s > \frac{3}{4} + \frac{1}{2r},
	\]
$L^q(0,T;B_{r}^{s- 3/2 \cdot \alpha(r)}(\mathbb R^2))
\hookrightarrow L^4(0,T;L^\infty(\mathbb R^2))$.
\end{proof}
%}}}

%{{{ Proof: Local well-posedness for $n=2$
\begin{proof}[Proof of Proposition \ref{Proposition:1.2}]
\quad \\
\textbf{Local well-posedness}
Let $(q_1,r_1)$ satisfy the condition of Lemma \ref{Lemma:3.2}, \eqref{eq:7},
and $q_1 > p-1$.
Let $q_2 = \infty$ and $r_2 = 2$.
Let $X^s(0,T) = L^\infty(0,T; H^s(\mathbb R^2))
\cap L^{q_1}(0,T;B_{r_1}^{s-3/2 \cdot \alpha(r_1)}(\mathbb R^2))$.
Let
	\[
	\Phi(u)(t)
	= U(t) u_0 - \int_0^t U(t-t') |u(t')|^{p-1} u(t') dt'.
	\]
Then, for $0 < T < 1$,
	\begin{align}
	\|\Phi(u)\|_{X^s(0,T)}
	&\leq \|u_0\|_{H^s(\mathbb R^2)}
	+ C \||u|^{p-1} u\|_{L^1(0,T;H^s(\mathbb R^2))}
	\label{eq:8}\\
	&\leq \|u_0\|_{H^s(\mathbb R^2)}
	+ C T^{1-(p-1)/q_1} \|u\|_{X^s(0,T)}^p,
	\nonumber
	\end{align}
and
	\begin{align*}
	&\|\Phi(u) - \Phi(v)\|_{X^s(0,T)}\\
	&\leq \| u_0 - v_0\|_{H^s(\mathbb R)}
	+ C \||u|^{p-1} u - |v|^{p-1} v\|_{L^1(0,T;H^s(\mathbb R^2))}
	\nonumber \\
	&\leq \| u_0 - v_0\|_{H^s(\mathbb R)}
	\nonumber \\
	&+ C T^{1-(p-1)/q_1} ( \|u\|_{X^s(0,T)} + \|v\|_{X^s(0,T)})^{p-1}
	\|u-v\|_{X^s (0,T)}
	\nonumber \\
	&+ C T^{1-(p-1)/q_1} ( \|u\|_{X^s(0,T)} + \|v\|_{X^s(0,T)})^{\max(1,p-1)}
	\|u-v\|_{X^s (0,T)}^{\min(1,p-1)}.
	\nonumber
	\end{align*}
This means if
$T < T_0
:= \big(
2^{1+\max(1,p-1)}C
( \|u_0\|_{H^s(\mathbb R^2)}^{p-1}
+ \|u_0\|_{H^s(\mathbb R^2)}^{\max(1,p-1)})
\big)^{- q_1/(q_1-p+1)}$,
then
$\Phi$ is a map from
	\[
	B_{X^s(0,T)}(2 \| u_0 \|_{H^s(\mathbb R^2)})
	:= \{ f \in X^s(0,T) \ | \ \|f\|_{X^s(0,T)} \leq 2 \|u_0\|_{H^s(\mathbb R^2)} \}.
	\]
to itself.
Moreover, if $p \geq 2$, $\Phi$ is a contraction map in $X^s(0,T)$.
If $p < 2$,
since for $z_1, z_0 \in \mathbb C$ with $|z_1| > |z_0|$,
	\begin{align*}
	&\big| |z_1|^{p-1} z_1 - |z_0|^{p-1} z_0 \big|\\
	&\leq |z_1|^{p-1} | z_1 - z_0|
	+ \frac{1}{p-1} \int_0^1 (\theta |z_1|+(1-\theta)|z_0|)^{p-2} |z_0| |z_1 -z_0|\\
	&\leq |z_1|^{p-1} | z_1 - z_0|
	+ \frac{1}{p-1} |z_0|^{p-1}|z_1 -z_0|,
	\end{align*}
then
	\begin{align}
	&\|\Phi(u) - \Phi(v)\|_{L^\infty(0,T; L^2(\mathbb R^2))}
	\label{eq:9}\\
	&\leq \| u_0 - v_0\|_{L^2(\mathbb R^2)}
	+ C \||u|^{p-1} u - |v|^{p-1} v\|_{L^1(0,T;L^2(\mathbb R^2))}
	\nonumber\\
	&\leq \| u_0 - v_0\|_{L^2(\mathbb R^2)}
	\nonumber\\
	&+ C T^{1-(p-1)/q_1} ( \|u\|_{X^s(0,T)} + \|v\|_{X^s(0,T)})^{p-1}
	\|u-v\|_{L^\infty (0,T; L^2(\mathbb R^2))}
	\nonumber
	\end{align}
and therefore $\Phi$ is a contraction map in $L^\infty(0,T;L^2(\mathbb R^2))$.
Let $u_1 \in B_{X^s(0,T)}(2 \| u_0 \|_{H^s(\mathbb R^2)})$
and $u_k = \Phi(u_{k-1})$ for $k \geq 2$.
Then
$(u_k)_{k=1}^\infty \subset B_{X^s(0,T)} (2 \| u_0 \|_{H^s(\mathbb R^2)})$
is a Cauchy sequence in $L^\infty(0,T;L^2(\mathbb R^2))$.
Let $u^\ast$ be the limit of $(u_k)_{k=1}^\infty$
in $L^\infty(0,T;L^2(\mathbb R^2))$.
Since
	\[
	\| |f|^{p-1} g \|_{L^{2/p}(\mathbb R^n)}
	\leq \| f \|_{L^{2}(\mathbb R^n)}^{p-1} \| g \|_{L^2(\mathbb R^n)}
	\]
for any $n \in \mathbb N$ and $f, g \in L^2(\mathbb R^n)$,
$\Phi(u_k) \to \Phi(u^\ast) $
in $L^\infty(0,T; L^2 (\mathbb R^2))$.
Therefore $u^\ast$ is a solution of \eqref{eq:4}.
Moreover, since
$X^s(0,T) \hookrightarrow L^\infty(0,T;H^s(\mathbb R^2))$,
$u^\ast$ is also in $L^\infty(0,T;H^s(\mathbb R^2))$.
By \eqref{eq:12},
	\[
	u^\ast \in L^{q_1}(0,T;B_{r_1}^{s-\frac{3}{2} \alpha(r_1)}(\mathbb R^2)).
	\]
If $s > 1$, by the Sobolev embedding,
for some $0 < \theta < 1$,
	\[
	\| u - v \|_{L^\infty(0,T;L^\infty(\mathbb R^2))}
	\lesssim \| u - v \|_{L^\infty(0,T;L^2(\mathbb R^2))}^\theta
	\| u - v \|_{L^\infty(0,T;H^1(\mathbb R^2))}^{1-\theta}
	\]
and therefore the solution map depends continuously on the initial data continuously
in $H^s(\mathbb R^2)$.
In the case where $s \leq 1$,
the continuous dependence of $\Phi$ may be shown as follows.
By \eqref{eq:9},
the solution map depends continuously on the initial data continuously
in $L^2(\mathbb R^2)$.
We define $s_3, s_4 > 0$ so that they satisfy the following:
$\max \bigg( \frac{3}{4} + \frac{1}{2r_1}, s_4-\frac{3}{4}(p-1) \bigg) < s_3 < s_4 < s$,
$r_3 = \frac{3}{2} \bigg( s_3-s_4+\frac{3}{4} \bigg)^{-1}$,
and $q_3 > p-1$,
where $(q_3,r_3)$ satisfy \eqref{eq:7}.
Then $B_{r_1}^{s_3-\frac{3}{2} \alpha(r_1)}(\mathbb R^2)
\hookrightarrow L^\infty(\mathbb R^2)$,
$B_{r_3'}^{s_3-\frac{3}{2} \alpha(r_3')}(\mathbb R^2)
= B_{r_3'}^{s_4}(\mathbb R^2)$,
$2 < r_3 < \frac{3}{2} \frac{4}{3(2-p)}$,
and $\frac{r_3(p-1)}{r_3-2} > 1$.
Moreover,
	\begin{align*}
	&\|\Phi(u) - \Phi(v)\|_{L^{q_1}(0,T;B_{r_1}^{s_3}(\mathbb R^2))}\\
	&\leq \| u_0 - v_0\|_{H^{s_3}(\mathbb R)}
	+ C \||u|^{p-1} u - |v|^{p-1} v\|_{L^{q_3'}(0,T;B_{r_3'}^{s_4}(\mathbb R^2))}.
	\end{align*}
For $z_j \in \mathbb C$ with $j=1,2,3,4$,
with $w_1=z_2-z_1$ and $w_2 = z_4-z_3$,
	\begin{align*}
	&|z_4|^{p-1} z_4 - |z_3|^{p-1} z_3 - |z_2|^{p-1} z_2 + |z_1|^{p-1} z_1\\
	&= \frac {p+1} 2 \int_0^1 |z_3 + \theta w_2 |^{p-1} d \theta w_2
	- \frac {p+1} 2 \int_0^1 |z_1 + \theta w_1 |^{p-1} d \theta w_1\\
	&+ \frac {p-1} 2 \int_0^1 |z_3 + \theta w_2 |^{p-3} (z_3 + \theta w_2 )^2 d \theta
	\overline{w_2}\\
	&- \frac {p-1} 2 \int_0^1 |z_1 + \theta w_1 |^{p-3} ( z_1 + \theta w_1 )^2 d \theta
	\overline{w_1}.
	\end{align*}
Then
	\begin{align*}
	&\bigg| \int_0^1 |z_3 + \theta w_2 |^{p-1} d \theta w_2
	- \int_0^1 |z_1 + \theta w_1 |^{p-1} d \theta w_1 \bigg|\\
	&\leq \int_0^1 |z_3 + \theta w_2 |^{p-1}
	d \theta | w_2 - w_1 |\\
	&+ \int_0^1 \big| |z_3 + \theta w_2 |^{p-1} - |z_1 + \theta w_1 |^{p-1} \big|
	d \theta | w_1 |\\
	&\leq (|z_3|^{p-1}+|z_4|^{p-1})|w_2-w_1|
	+ |z_3-z_1|^{p-1} |w_1| + \frac{1}{p} |w_1| |w_2-w_1|^{p-1}\\
	&\leq (|z_3|^{p-1}+|z_4|^{p-1})|w_2-w_1|
	+ \frac{p+1}{p} |w_1| |z_3-z_1|^{p-1} + \frac{1}{p} |w_1| |z_4-z_2|^{p-1}.
	\end{align*}
Similarly
	\begin{align*}
	&\bigg| \int_0^1 |z_3 + \theta w_2 |^{p-3} (z_3 + \theta w_2 )^2 d \theta
	\overline{w_2}
	- \int_0^1 |z_1 + \theta w_1 |^{p-3} ( z_1 + \theta w_1 )^2 d \theta
	\overline{w_1} \bigg|\\
	&\lesssim (|z_3|^{p-1}+|z_4|^{p-1})|w_2-w_1|
	+ |w_1| |z_3-z_1|^{p-1} + |w_1| |z_4-z_2|^{p-1},
	\end{align*}
since
	\begin{align*}
	| |z_2|^{p-3} z_2^2 - |z_1|^{p-3} z_1^2 |
	&\lesssim \int_0^1 |z_1^2 + \theta (z_2-z_1)|^{p-2} d \theta |z_2-z_1|\\
	&= \int_0^1 |z_1^2 + \theta (z_2-z_1)|^{p-2} d \theta |z_2-z_1|\\
	&\leq \int_0^1 ||z_1| - \theta |z_2-z_1||^{p-2} d \theta |z_2-z_1|\\
	&= \frac{1}{2-p} ( |z_1|^{p-1} - ||z_1| - |z_2-z_1| |^{p-1} )\\
	&\leq |z_2-z_1|^{p-1}.
	\end{align*}
Therefore,
	\begin{align*}
	&\big\| |u(t,\cdot+h)|^{p-1} u(t,\cdot+h) - |v(t,\cdot+h)|^{p-1} v(t,\cdot+h)\\
	&\quad
	- |u(t)|^{p-1} u(t) + |v(t)|^{p-1} v(t) \big\|_{L^{r_3'}(\mathbb R^2)}\\
	&= \big\| |u(t,\cdot+h)|^{p-1} u(t,\cdot+h)- |u(t)|^{p-1} u(t)\\
	&\quad - |v(t,\cdot+h)|^{p-1} v(t,\cdot+h) + |v(t)|^{p-1} v(t)
	\big\|_{L^{r_3'}(\mathbb R^2)}\\
	&\leq 4 \| u(t) \|_{L^{\frac{2r_3(p-1)}{r_3-2}}(\mathbb R^2)}^{p-1}
	\| u(t,\cdot+h) - v(t,\cdot+h) - u(t) + v(t) \|_{L^{2}(\mathbb R^2)}\\
	&+ \frac{2(p+2)}{p} \| v(t,\cdot+h) - v(t) \|_{L^{2}(\mathbb R^2)}
	\| u(t) - v(t) \|_{L^{\frac{2r_3(p-1)}{r_3-2}}(\mathbb R^2)}^{p-1}
	\end{align*}
and this means
	\begin{align*}
	&\||u|^{p-1} u - |v|^{p-1} v\|_{L^{q_3'}(0,T;B_{r_3'}^{s_4}(\mathbb R^2))}\\
	&\lesssim \big\| \| u \|_{L^{\frac{2r_3(p-1)}{r_3-2}}(\mathbb R^2)}^{p-1}
	\| u - v \|_{H^{s_4}(\mathbb R^2)}
	+ \| v \|_{H^{s_4}(\mathbb R^2)}
	\| u - v \|_{L^{\frac{2r_3(p-1)}{r_3-2}}(\mathbb R^2)}^{p-1} \|_{L^{q_3'}(0,T)}\\
	&\leq \big\|
	\| u \|_{L^{\infty}(\mathbb R^2)}^{p-1-\frac{r_3-2}{r_3}}
	\| u \|_{L^{2}(\mathbb R^2)}^{\frac{r_3-2}{r_3}}
	\| u - v \|_{H^{s_4}(\mathbb R^2)}\|_{L^{q_3'}(0,T)}\\
	&+ \big\| \| v \|_{H^{s_4}(\mathbb R^2)}
	\| u - v \|_{L^{\infty}(\mathbb R^2)}^{p-1-\frac{r_3-2}{r_3}}
	\| u - v \|_{L^{2}(\mathbb R^2)}^{\frac{r_3-2}{r_3}}
	\|_{L^{q_3'}(0,T)}\\
	&\leq
	\| u \|_{L^{q_3'(p-1-\frac{r_3-2}{r_3})}(0,T;L^{\infty}(\mathbb R^2))}%
	^{p-1-\frac{r_3-2}{r_3}}
	\| u \|_{L^{\infty}(0,T;L^{2}(\mathbb R^2))}^{\frac{r_3-2}{r_3}}
	\| u - v \|_{L^\infty(0,T;H^{s_4}(\mathbb R^2))}\\
	&+ \| v \|_{L^\infty(0,T;H^{s_4}(\mathbb R^2))}
	\| u - v \|_{L^{q_3'(p-1-\frac{r_3-2}{r_3})}(0,T;L^{\infty}(\mathbb R^2))}%
	^{p-1-\frac{r_3-2}{r_3}}
	\| u - v \|_{L^\infty(0,T;L^{2}(\mathbb R^2))}^{\frac{r_3-2}{r_3}}.
	\end{align*}
Since
	\begin{align*}
	&\| \Phi(u) - \Phi(v) \|_{L^\infty(0,T;H^s(\mathbb R^2))}\\
	&\lesssim \| u_0 - v_0 \|_{H^s(\mathbb R^2)}
	+ ( \|u_0\|_{H^s(\mathbb R^2)} + \|v_0\|_{H^s(\mathbb R^2)} )
	\| u - v \|_{L^{p-1}(0,T;L^\infty(\mathbb R^2))}^{p-1},
	\end{align*}
the solution map is also continuously dependent
in $L^\infty(0,T;H^s(\mathbb R^2))$.

\textbf{Global well-posedness}
When $s=1$ and when $s=2$ and $p=3$,
a priori estimates shows the global well-posedness
by the blow-up alternative argument.
Here we consider the case where $p=3$ and $1 < s < 2$.
Let $[a]$ be the highest integer which is less than or equal to $a$.
Let $T_1 = \min(1,T_0)$.
By using the $H^1$ a priori estimate,
for any $t > 0$,
	\begin{align*}
	\|u\|_{L^4(0,t;L^\infty(\mathbb R^2))}
	&\leq \sum_{k=0}^{[t/T_1]+1} \|u\|_{L^4(k T_1, (k+1)T_1; L^\infty(\mathbb R^2))}\\
	&\leq \sum_{k=0}^{[t/T_1]+1} \|u\|_{X^1(kT_1,(k+1)T_1)}\\
	&\leq 2 T_1^{-1} (1+t) \|u_0\|_{H^1(\mathbb R^2)}.
	\end{align*}
Then by using Proposition \ref{Proposition:2.3},
	\begin{align*}
	\|u(t)\|_{\dot H^s(\mathbb R^2)}^2
	&\lesssim \|u_0\|_{H^s(\mathbb R^2)}^2
	+ \int_{0}^{t} \|u(t')\|_{L^{\infty}(\mathbb R^n)}^{p-1}
	\| u(t') \|_{\dot H^s(\mathbb R^2)}^2 dt\\
	&\lesssim \|u_0\|_{H^s(\mathbb R^2)}^2
	+ \|u(t')\|_{L^{2(p-1)}(0,t;L^{\infty}(\mathbb R^2))}^{p-1}
	\| u \|_{L^4(0,t;\dot H^s(\mathbb R^2))}^2\\
	&\lesssim \|u_0\|_{H^s(\mathbb R^2)}^2
	+ \|u_0\|_{H^1(\mathbb R^2)}^{p-1} (1+t)^{p-1} t^{\frac{1}{2} - \frac{p-1}{4}}
	\| u \|_{L^4(0,t;\dot H^s(\mathbb R^2))}^2.
	\end{align*}
This shows
	\[
	\|u(t)\|_{\dot H^s(\mathbb R^2)}^4
	\lesssim \|u_0\|_{H^s(\mathbb R^2)}^4
	+ \|u_0\|_{H^1(\mathbb R^2)}^{2(p-1)} (1+t)^{2(p-1)} t^{1 - \frac{p-1}{2}}
	\| u \|_{L^4(0,t;\dot H^s(\mathbb R^2))}^4.
	\]
By the Gronwall inequality,
	\[
	\|u(t)\|_{H^s(\mathbb R^2)}
	\lesssim \|u_0\|_{H^s(\mathbb R^2)}
	\exp \{ C \|u_0\|_{H^1(\mathbb R^2)}^{2(p-1)} (1+t)^{2(p-1)} t^{2-\frac{p-1}{2}} \}.
	\]
This shows the global well-posedness in $H^s(\mathbb R^2)$ setting.
\end{proof}
%}}}
%}}}

%{{{ subsection: md case
\subsection{The case where $n \geq 3,$ global $H^1$ existence result}
%{{{ In the case where $n \geq 3$,
In the case where $n \geq 3$,
the Strichartz estimate Lemma \ref{Lemma:3.1} doesn't seem sufficient
to obtain a uniform control of solutions in the $H^1(\mathbb R^3)$ setting.
So here, we consider radial data and use the following Strauss lemma.
%}}}

%{{{ Lemma : bounded
\begin{Lemma}[{\cite[Theorems 1,2]{SS}}, {\cite[Proposition 1]{CO}}]
\label{Lemma:3.3}
Let $n \geq 2$ and let $1/2 < s < n/2$.
Then
	\[
	\| |x|^{\frac{n}{2} - s} f \|_{L_{\mathrm{rad}}^\infty(\mathbb R^n)}
	\lesssim \| f \|_{\dot H_{\mathrm{rad}}^s(\mathbb R^n)}.
	\]
\end{Lemma}
%}}}

%{{{ Since solutions are not uniformly controlled at the origin
Since solutions are not uniformly controlled at the origin
by the Strauss lemma above,
we apply the following weighted Strichartz estimate:
%}}}

%{{{ Lemma : xd
\begin{Lemma}[{\cite[Propositions 1.2 and 1.3]{BGV}}]
\label{Lemma:3.4}
Let $n \in \mathbb N$.
Let $\delta > 0$ and $[x]_{\delta} = |x|^{1-\delta} + |x|^{1+\delta}$.
The for any $q_1 \in \lbrack 2,\infty)$ and $q_2 \in (2,\infty)$,
	\begin{align*}
	\| [x]_{\delta}^{-1/q_1} U(t) f \|_{L^{q_1}(\mathbb R;L^2(\mathbb R^n))}
	&\lesssim \| f \|_{L^2(\mathbb R^n)},\\
	\bigg\| [x]_{\delta}^{-1/q_1} \int_0^t U(t-t') F(t') dt'
	\bigg\|_{L^{q_1}(0,T;L^2(\mathbb R^n))}
	&\lesssim \| [x]_{\delta}^{1/q_2} F \|_{L^{q_2'}(0,T;L^2(\mathbb R^n))}.
	\end{align*}
\end{Lemma}
%}}}

%{{{ Proof: Local well-posedness for $n \geq 3$
\begin{proof}[Proof of Proposition \ref{Proposition:1.3}]
By using the uniform $H^1(\mathbb R^n)$ control obtained in \eqref{eq:5},
we reduce the proof to the local well-posedness in $H^1(\mathbb R^n)$.
Let $\delta > 0$, $1/2 < s < 1$, and $2 < q_1, q_2 < \infty$ satisfy
	\begin{align}
	- (p-1) \bigg( \frac n 2 - s \bigg) + \frac{1-\delta}{q_2}
	&= - \frac{1-\delta}{q_1}.
	\label{eq:10}
	\end{align}
We remark that there exist $\delta, q_1, q_2,r$
if $1 < p < 1 + 2/(n-2)$ since,
	\begin{align*}
	(p-1) \bigg( \frac n 2 - s \bigg) < 1
	\Longrightarrow
	p
	< 1 + \frac{2}{n-2s}
	< 1 + \frac{2}{n-2}.
	\end{align*}
We define the norm $Y^1(T)$ as
	\begin{align*}
	\| u \|_{Y^1(T)}
	&= \|u\|_{L^\infty(0,T;H_{\mathrm{rad}}^1(\mathbb R^n))}\\
	&+ \|[x]_{\delta}^{-1/q_1} u\|_{L^{q_1}(0,T;L_{\mathrm{rad}}^2(\mathbb R^n))}
	+ \|[x]_{\delta}^{-1/q_1} \nabla u\|_{L^{q_1}(0,T;L_{\mathrm{rad}}^2(\mathbb R^n))}.
	\end{align*}
Let $\psi \in \mathcal S (\mathbb R^n)$ satisfy
	\[
	\psi(x) =
	\begin{cases}
	1 & \mathrm{if}\quad |x| \leq 1,\\
	0 & \mathrm{if}\quad |x| \geq 2.
	\end{cases}
	\]
Then by Lemmas \ref{Lemma:3.3} and \ref{Lemma:3.4}
and \eqref{eq:10},
	\begin{align*}
	&\| \Phi (u) \|_{Y^1(T)}\\
	&\lesssim \| u_0 \|_{H_{\mathrm{rad}}^1(\mathbb R^n)}
	+ \bigg\| [x]_{\delta}^{-1/q_1}
	\int_0^t U(t-t') \big( \psi |u(t')|^{p-1} u(t') \big) dt'
	\bigg\|_{Y^1(T)}\\
	&+ \bigg\| [x]_{\delta}^{-1/q_1} \int_0^t U(t-t')
	\big( (1-\psi)|u(t')|^{p-1} u(t') \big) dt' \bigg\|_{Y^1(T)}\\
	&\lesssim \| u_0 \|_{H_{\mathrm{rad}}^1(\mathbb R^n)}\\
	&+ \| |x|^{- (p-1) ( \frac n 2 - s ) + \frac{1-\delta}{q_2}}
	||x|^{\frac n 2 - s} u|^{p-1} u
	\|_{L^{q_2'}(0,T;L_{\mathrm{rad}}^2(|x|\leq 2))}\\
	&+ \| |x|^{- (p-1) ( \frac n 2 - s ) + \frac{1-\delta}{q_2}}
	||x|^{\frac n 2 -s}u|^{p-1} \nabla u
	\|_{L^{q_2'}(0,T;L_{\mathrm{rad}}^2(|x|\leq 2))}\\
	&+ \| |u|^{p-1}u \|_{L^1(0,T;L_{\mathrm{rad}}^2(|x|>1))}
	+ \| \nabla (|u|^{p-1}u) \|_{L^1(0,T;L_{\mathrm{rad}}^2(|x|>1))}\\
	&\lesssim \| u_0 \|_{H_{\mathrm{rad}}^1(\mathbb R^n)}
	+ T^{1-\frac{1}{q_1}-\frac{1}{q_2}}
	\| u \|_{Y^1(T)}^{p}
	\end{align*}
and therefore for some $T$ and $R$,
$\Phi$ is a map from $B_{Y^1(T)}(R)$
into itself.
Moreover,
	\begin{align}
	&\| \Phi(u) - \Phi(v) \|_{Y^1(T)}
	\label{eq:11}\\
	&\lesssim \| u_0 - v_0 \|_{H_{\mathrm{rad}}^1(\mathbb R^n)}
	\nonumber\\
	&+ \| |x|^{- \frac{1-\delta}{q_1}}
	(||x|^{\frac n 2 - s} u|^{p-1} - ||x|^{\frac n 2 - s} v|^{p-1})
	(|\nabla u| + |u|) \|_{L^{q_2'}(0,T;L_{\mathrm{rad}}^2(|x|\leq 2))}
	\nonumber\\
	&+ \| |x|^{- \frac{1-\delta}{q_1}} | |x|^{\frac n 2 - s} v|^{p-1}
	(|\nabla (u-v)| + |u-v|) \|_{L^{q_2'}(0,T;L_{\mathrm{rad}}^2(|x| \leq 2))}
	\nonumber\\
	&+ \| (| |x|^{\frac n 2 - s} u|^{p-1}
	- | |x|^{\frac n 2 -s} v|^{p-1}) |x|^{-\frac{1+\delta}{q_1}} (|\nabla u| + |u|)
	\|_{L^{1}(0,T;L_{\mathrm{rad}}^2(|x| > 1))}
	\nonumber\\
	&+ \| | |x|^{\frac n 2 - s} v|^{p-1}
	|x|^{-\frac{1+\delta}{q_1}} (|\nabla (u-v)| + |u-v|)
	\|_{L^{1}(0,T;L_{\mathrm{rad}}^2(|x| > 1))}.
	\nonumber
	\end{align}
Then for $ p \geq 2$, $\Phi$ is a contraction map on $B_{Y^1(T)}(R)$.
Similarly, for $1 < p < 2$,
we define the auxiliary norm $Y^0(T)$ as
	\[
	\| u \|_{Y^0(T)}
	= \|u\|_{L^\infty(0,T;L_{\mathrm{rad}}^2(\mathbb R^n))}
	+ \|[x]_{\delta}^{-1/q_1} u\|_{L^q(0,T;L_{\mathrm{rad}}^2(\mathbb R^n))}.
	\]
Then for $1 < p < 2$,
	\begin{align*}
	&\| (\Phi(u) - \Phi(v) ) \|_{Y^0(T)}\\
	&\lesssim \| u_0 - v_0 \|_{L^2(\mathbb R^n)}\\
	&+ \Big\| [x]_{\delta}^{-1/q_1}\big( \big| |x|^{\frac n 2 - s} v \big|
	+ \big| |x|^{\frac n 2 - s} v \big| \big)^{p-1}
	|u-v| \|_{L^{q_2'}(0,T;L_{\mathrm{rad}}^2(|x| \leq 2))}\\
	&+ \Big\| \big( \big| |x|^{\frac n 2 - s} v \big|
	+ \big| |x|^{\frac n 2 - s} v \big| \big)^{p-1}
	|u-v| \|_{L^{1}(0,T;L_{\mathrm{rad}}^2(|x| > 1))}\\
	&\lesssim T^{1-\frac{1}{q_2}}
	(\|u\|_{Y^1(T)} + \|v\|_{Y^1(T)})^{p-1}
	\| u - v \|_{Y^0(T)}.
	\end{align*}
and therefore $\Phi$ is a contraction map in
$Y^0(T)$ for some $T$ and $R$ and therefore
we have a unique solution to \eqref{eq:1} in $Y^1(T)$.
Moreover,
by Lemma \ref{Lemma:3.3} and \eqref{eq:11},
with some $0 < \theta < 1$,
	\begin{align*}
	&\| u - v \|_{Y^1(T)}\\
	&\lesssim \|u_0-v_0\|_{H_{\mathrm{rad}}^1(\mathbb R^n)}
	+ T^{1-\frac{1}{q_1}-\frac{1}{q_2}}
	(\| u \|_{Y^1(T)} + \| v \|_{Y^1(T)} )^{p-1} \| u - v \|_{Y^1(T)}\\
	&+ T^{1-\frac{1}{q_1}-\frac{1}{q_2}}
	(\| u \|_{Y^1(T)} + \| v \|_{Y^1(T)} )
	\big\| |x|^{\frac n 2 - s}(u - v)
	\big\|_{L^\infty(0,T;L_{\mathrm{rad}}^\infty(\mathbb R^n)}^{p-1}\\
	&\lesssim \|u_0-v_0\|_{H_{\mathrm{rad}}^1(\mathbb R^n)}
	+ T^{1-\frac{1}{q_1}-\frac{1}{q_2}}
	(\| u \|_{Y^1(T)} + \| v \|_{Y^1(T)} )^{p-1} \| u - v \|_{Y^1(T)}\\
	&+ T
	(\| u \|_{Y^1(T)} + \| v \|_{Y^1(T)} )
	\big\| u - v \big\|_{Y^0(T)}^{p-1}
	\end{align*}
and therefore $\| u - v \|_{Y^1(T)} \to 0$
as $\| u_0 - v_0\|_{H_{\mathrm{rad}}^1(\mathbb R^n)} \to 0$.
\end{proof}
%}}}
%}}}

%{{{ subsection: 3d case
\subsection{3 dimensional case, small  $H^1$ data solutions for $p=3$}
%{{{ In the three dimensional case,
In the three dimensional scaling critical case,
the weighted Strichartz estimate Lemma \ref{Lemma:3.4} doesn't seem sufficient
to control solutions uniformly.
So here, we transform \eqref{eq:1} into the corresponding wave equation.
%}}}

%{{{ The Cauchy problem \eqref{eq:1} with initial data $u(0) = u_0$
The Cauchy problem \eqref{eq:1} with initial data $u(0) = u_0$
is rewritten as the following:
	\begin{align*}
	\Box u &= -i ( - i \partial_t - D ) |u|^{p-1} u\\
	&= i \frac{p+1}{2} |u|^{p-1} ( D u - i |u|^{p-1} u)\\
	&- i \frac{p-1}{2} |u|^{p-3} u^2 \overline{(D u - i |u|^{p-1} u)}
	+ i D ( |u|^{p-1} u )\\
	&- i \bigg( D (|u|^{p-1} u) + \frac{p+1}{2} |u|^{p-1} D u
	- \frac{p-1}{2} |u|^{p-3} u^2 D \overline u \bigg)
	+ p |u|^{2p-2} u\\
	&=: F_p(u).
	\end{align*}
Then the corresponding integral equation is the following:
	\begin{align}
	u(t)
	&= \cos(t D) u_0 + \frac{\sin(tD)}{D}(-i Du_0 - |u_0|^{p-1} u_0)
	\label{eq:12}\\
	&+ \int_0^t \frac{\sin((t-t')D)}{D} F_p(u)(t') dt'.
	\nonumber
	\end{align}
For any radially symmetric function $f$,
we define $\widetilde f$ as $\widetilde f(|x|) = f(x)$.
Then for any radial data,
\eqref{eq:12} is rewritten as
	\begin{align}
	&\widetilde u(t)
	\label{eq:13}\\
	&= \partial_t J[u_0](t) + J[-i Du_0 - |u_0|^{p-1} u_0](t)
	+ \int_0^t J[ F_p(u)(t')](t-t') dt'
	\nonumber
	\end{align}
where
	\[
	J[f](t,r) = \frac{1}{2r} \int_{|r-t|}^{r+t} \lambda \widetilde f(\lambda) d \lambda.
	\]
This transformation is justified as follows:
%}}}

%{{{ Lemma: Relationship
\begin{Lemma}
\label{Lemma:3.5}
Let $1 < p \leq 3$ and $u_0 \in H_{\mathrm{rad}}^1(\mathbb R^3)$
and $u \in C(0,T;H_{\mathrm{rad}}^1(\mathbb R^3))$ be the solution of \eqref{eq:4}.
Then $u$ is also the solution of \eqref{eq:13}
\end{Lemma}

\begin{proof}
Since $H^1(\mathbb R^3) \hookrightarrow L^{2p}(\mathbb R^3)$,
$u \in C^1(0,T;L_{\mathrm{rad}}^2(\mathbb R^3))$
and $|u|^{p-1} u \in C^1(0,T;H_{\mathrm{rad}}^{s}(\mathbb R^3))$
with $s < - 3/2$.
Then $- i ( - i \partial_t - D) |u|^{p-1} u = F_p(u)$
holds in the $H_{\mathrm{rad}}^{s}(\mathbb R^3)$ setting.
Moreover,
	\[
	U(t) u_0 = \cos(tD) u_0 - i \sin(tD) u_0
	\]
and in the $H_{\mathrm{rad}}^{s}(\mathbb R^3)$ setting,
the following calculation is also justified:
	\begin{align*}
	& - \int_0^t U(t-t') ( |u(t')|^{p-1} u(t') ) dt'\\
	&= - \int_0^t \cos((t-t')D) ( |u(t')|^{p-1} u(t') ) dt'\\
	&- \int_0^t -i \sin((t-t')D) ( |u(t')|^{p-1} u(t') ) dt'\\
	&= \bigg[ \frac{\sin((t-t')D)}{D} ( |u(t')|^{p-1} u(t') ) \bigg]_{t'=0}^t\\
	&- \int_0^t \frac{\sin((t-t')D)}{D} \partial_t ( |u(t')|^{p-1} u(t') ) dt'\\
	&- \int_0^t -i \sin((t-t')D) ( |u(t')|^{p-1} u(t') ) dt'\\
	&= - \frac{\sin(tD)}{D} ( |u_0|^{p-1} u_0 )
	- i \int_0^t \frac{\sin((t-t')D)}{D} ( - i \partial_t - D )( |u(t')|^{p-1} u(t') ) dt'.
	\end{align*}
Therefore $u$ is also a solution of \eqref{eq:13}.
\end{proof}
%}}}

%{{{ To obtain the uniform control,
To obtain the uniform control,
we use the estimates below regarding $J$.
For any $f: [0,\infty) \to \mathbb C$,
we define $A[f] : \mathbb R \to \mathbb C$ as
$A[f](\lambda) = f(|\lambda|)$.
See also \cite{KM}.
%}}}

%{{{ Lemma: L boundedness
\begin{Lemma}
\label{Lemma:3.6}
Let $f : [0,\infty) \to \mathbb C$.
Then
	\[
	\bigg\| \frac 1 {2 \cdot}
	\int_{|\cdot -t|}^{\cdot +t} f(\lambda) d \lambda \bigg\|_{L^\infty(0,\infty)}
	\leq M[A[f]](t).
	\]
\end{Lemma}

\begin{proof}
If $r \leq t$,
then
	\[
	\frac{1}{2r} \int_{t-r}^{t+r} |f(\lambda)| d \lambda
	\leq M[f](t).
	\]
If $r > t$,
then
	\[
	\frac{1}{2r} \int_{r-t}^{t+r} |f(\lambda)| d \lambda
	\leq \frac{1}{2r} \int_{t-r}^{t+r} |A[f](\lambda)| d \lambda
	\leq M[A[f]](t).
	\]

\end{proof}
%}}}

%{{{ Corollary: L boundedness
\begin{Corollary}
\label{Corollary:3.7}
Let $f : \mathbb R^3 \to \mathbb C$ be radial.
Then
	\[
	\| J[f] \|_{L^2(0,T;L^\infty(\mathbb R^3))}
	\leq C \|f\|_{L_{\mathrm{rad}}^2(\mathbb R^3)}.
	\]
\end{Corollary}

\begin{proof}
Let $g(\lambda) = \lambda \widetilde f (\lambda)$.
Then
	\[
	\| J[f] \|_{L^2(0,T;L^\infty(0,\infty))}
	\leq \| M[A[g]] \|_{L^2(0,T)}
	\leq C \| g \|_{L^2(0,\infty)}
	= C \| f \|_{L_{\mathrm{rad}}^2(\mathbb R^3)}.
	\]
\end{proof}
%}}}

%{{{ Corollary: L boundedness
\begin{Corollary}
\label{Corollary:3.8}
Let $h : [0,\infty) \times \mathbb R^3 \to \mathbb C$ be radial.
Then
	\[
	\bigg\| \int_0^t J[h(t')](t-t') dt' \bigg\|_{L^2(0,T;L^\infty(0,\infty))}
	\leq C \|h\|_{L^1(0,T;L_{\mathrm{rad}}^2(\mathbb R^3))}.
	\]
\end{Corollary}

\begin{proof}
Let $H(t,\lambda) = \lambda \widetilde h (t,\lambda)$.
Then
	\begin{align*}
	\bigg\| \int_0^t J[h(t')](t-t') dt' \bigg\|_{L^2(0,T;L^\infty(0,\infty))}
	&\leq \bigg\| \int_0^t M[A[H(t')]](t-t') dt' \bigg\|_{L^2(0,T)}\\
	&\leq \int_0^T \| M[A[H(t')]]\|_{L^2(t',T)} dt'\\
	&\leq C \int_0^T \| h(t')\|_{L_{\mathrm{rad}}^2(\mathbb R^3)} dt'.
	\end{align*}
\end{proof}
%}}}

%{{{ Corollary: L boundedness
\begin{Corollary}[Hardy]
\label{Corollary:3.9}
Let $f \in C^1(\mathbb R;\mathbb C)$.
Then
	\[
	\bigg\| \frac{d}{dt} \bigg( \frac{1}{2r}
	\int_{|r-t|}^{r+t} \lambda f(\lambda) d \lambda \bigg) \bigg\|_{L^2(0,\infty;L^\infty(0,\infty))}
	\leq C \| r f'\|_{L^2(0,\infty)}.
	\]
\end{Corollary}
%}}}

%{{{ Proof: Local well-posedness for $n=3$
\begin{proof}[Proof of Proposition \ref{Proposition:1.4}]
For $0 < T < 1$ and $p=3$,
By Corollaries \ref{Corollary:3.7}, \ref{Corollary:3.8}, and \ref{Corollary:3.9},
	\begin{align}
	&\|u\|_{L^2(0,T;L_{\mathrm{rad}}^\infty(\mathbb R^3))}
	\label{eq:14}\\
	&\lesssim \|u_0\|_{H_{\mathrm{rad}}^1(\mathbb R^3)}
	+ \|u_0\|_{H_{\mathrm{rad}}^1(\mathbb R^3)}^3
	+ \|F_3\|_{L^1(0,T;L_{\mathrm{rad}}^2(\mathbb R^3))}
	\nonumber\\
	&\lesssim \|u_0\|_{H_{\mathrm{rad}}^1(\mathbb R^3)}
	+ \|u_0\|_{H_{\mathrm{rad}}^1(\mathbb R^3)}^3
	\nonumber\\
	&+ \|u\|_{L^{2}(0,T;L_{\mathrm{rad}}^\infty(\mathbb R^3)}^{2}
	\|u\|_{L^\infty(0,T;H_{\mathrm{rad}}^1(\mathbb R^3))}
	+ \| |u|^{5} \|_{L^1(0,T;L_{\mathrm{rad}}^2(\mathbb R^3))}.
	\nonumber
	\end{align}
Since
	\[
	\| |u|^{5} \|_{L^1(0,T;L_{\mathrm{rad}}^2(\mathbb R^3))}
	\leq \| u \|_{L^{2}(0,T;L_{\mathrm{rad}}^\infty(\mathbb R^3))}^{2}
	\|u\|_{L^\infty(0,T;H_{\mathrm{rad}}^1(\mathbb R^3))}^3,
	\]
by the unitarity of $U(t)$,
	\[
	\|u\|_{L^\infty(0,T;H_{\mathrm{rad}}^1(\mathbb R^3))}
	\lesssim \|u_0\|_{H_{\mathrm{rad}}^1(\mathbb R^3)}
	+ \|u\|_{L^{p-1}(0,T;L_{\mathrm{rad}}^\infty(\mathbb R^3))}^{p-1}
	\|u_0\|_{H_{\mathrm{rad}}^1(\mathbb R^3)}.
	\]
Let
$X_{\mathrm{rad}}^1(0,T) =
L^\infty(0,T;H_{\mathrm{rad}}^1(\mathbb R^3))
\cap L^2(0,T;L_{\mathrm{rad}}^\infty(\mathbb R^3))$.
Then,
for sufficiently small initial data $u_0$,
$\Phi$ maps from $B_{X_{\mathrm{rad}}^1(0,T)}(R)$
into itself with some $T$ and $R$.
Moreover,
	\begin{align*}
	&\||u|^{2} u - |v|^{2} v \|_{L^1(0,T;H_{\mathrm{rad}}^1(\mathbb R^3))}\\
	&\lesssim
	(\|u\|_{X_{\mathrm{rad}}(0,T)} + \|v\|_{X_{\mathrm{rad}}(0,T)} )^{2}
	\|u-v\|_{L^\infty(0,T;H_{\mathrm{rad}}^1(\mathbb R^3))}.
	\end{align*}
Since
	\begin{align*}
	&| F(u) - F(v)|\\
	&= \Big| i \big( D (|u|^{2p-1} u) - 2 |u|^{2} D u
	- u^2 D \overline u \big)
	+ 3 |u|^{4} u\\
	&- i \big( D (|v|^{2} v) - 2|v|^{2} D v
	- v^2 D \overline v \big)
	- 3 |v|^{4} v \Big|\\
	&\leq | D ( |u|^{2} u - |v|^{2} v) |
	+ |u|^{2} | D (u-v) |\\
	&+ \Big( \big| |u|^{2} - |v|^{2} \big|
	+ \big| u^2 - v^2 \big| \Big) | D v |
	+ 3 \big| |u|^{4} u - |v|^{4} v \big|,
	\end{align*}
we have
	\begin{align*}
	&\|F(u) - F(v)\|_{L^1(0,T;L_{\mathrm{rad}}^2(\mathbb R^3))}\\
	&\lesssim
	(\|u\|_{X_{\mathrm{rad}}^1(0,T)} + \|v\|_{X_{\mathrm{rad}}^1(0,T)})^{2}
	\|u-v\|_{X_{\mathrm{rad}}^1(0,T)}\\
	&+
	(\|u\|_{X_{\mathrm{rad}}^1(0,T)} + \|v\|_{X_{\mathrm{rad}}^1(0,T)})^{4}
	\|u-v\|_{X_{\mathrm{rad}}^1(0,T)}.
	\end{align*}
This means for sufficiently small $u_0$,
$\Phi$ is a contraction map on $B_{X_{\mathrm{rad}}^1(0,T)}(R)$.
\end{proof}
%}}}
%}}}
%}}}

%{{{ bibliograpy

%}}}

%}}}
\end{document}